\documentclass[a4paper,notitlepage,twoside,leqno,12pt]{amsart}

\usepackage{mathrsfs}
\usepackage{bbm,pifont,latexsym}
\usepackage{anysize}  
\marginsize{2.8cm}{2.8cm}{2.3cm}{2.3cm}

\usepackage{dcolumn,indentfirst}
\usepackage[pdfpagemode=UseNone,linktocpage=true,pdfstartview=FitH]{hyperref}
\usepackage{amsmath,amssymb,amscd,amsthm,amsfonts,mathrsfs}
\usepackage{color,graphicx,xcolor,graphics,subfigure,extarrows,caption2}
\usepackage{titletoc} 

\newtheorem{thm}{Theorem}[section]
\newtheorem{lema}[thm]{Lemma}
\newtheorem{cor}[thm]{Corollary}
\newtheorem{prop}[thm]{Proposition}

\theoremstyle{definition}
\newtheorem*{defi}{Definition}
\newtheorem*{rmk}{Remark}
\newtheorem{exam}{Example}

\newcommand{\D}{\mathbb{D}}

\newcommand{\R}{\mathbb{R}}

\newcommand{\C}{\mathbb{C}}
\newcommand{\EC}{\widehat{\mathbb{C}}}

\newcommand{\MA}{\mathcal{A}}

\newcommand{\ii}{\textup{i}}

\makeatletter\@addtoreset{equation}{section}\makeatother

\begin{document}

\author{YINGQING XIAO}
\address{College of Mathematics and Econometrics, Hunan University, Changsha, 410082, P. R. China}
\email{ouxyq@hnu.edu.cn}

\author{FEI YANG}
\address{Department of Mathematics, Nanjing University, Nanjing, 210093, P. R. China}
\email{yangfei\rule[-2pt]{0.2cm}{0.5pt}math@163.com}

\title[SINGULAR PERTURBATIONS OF THE SIMPLE POLYNOMIALS]
{SINGULAR PERTURBATIONS WITH MULTIPLE POLES OF THE SIMPLE POLYNOMIALS}

\begin{abstract}
In this article, we study the dynamics of the following family of rational maps with one parameter:
\begin{equation*}
f_\lambda(z)= z^n+\frac{\lambda^2}{z^n-\lambda},
\end{equation*}
where $n\geq 3$ and $\lambda\in\mathbb{C}^*$. This family of rational maps can be viewed as a singular perturbations of the simple polynomial $P_n(z)=z^n$. We give a characterization of the topological properties of the Julia sets of the family $f_\lambda$ according to the dynamical behaviors of the orbits of the free critical points.
\end{abstract}



\date{\today}



\maketitle


\section{Introduction}

For a given rational map $f:\EC\rightarrow \EC$, we are interested in the dynamical system generated by the iterates of $f$. In this setup, the Riemann sphere $\EC$ can be divided into two dynamically meaningful and completely invariant subsets: the Fatou set and the Julia set. The \emph{Fatou set} $F(f)$ of $f$ is defined to be the set of points at which the family of iterates of $f$ forms a normal family, in the sense of Montel. The complement of the Fatou set is called the \emph{Julia set}, which we denote by $J(f)$. A connected component of the Fatou set is called a \emph{Fatou component}. According to Sullivan's theorem, every Fatou component of a rational map is eventually periodic and there are five kinds of periodic Fatou components: attracting domains, super-attracting domains, parabolic domains, Siegel disks and Herman rings.

The topology of the Julia sets of rational maps, such as the connectivity and local connectivity, is an interesting and important problem in complex dynamics. It was proved by Fatou that the Julia set of a polynomial is connected if and only if the polynomial has bounded critical orbits. Recently, Qiu and Yin, independently, Kozlovski and van Strien, gave a sufficient and necessary condition for the Julia set of a polynomial to be a Cantor set and hence gave an affirmative answer to Branner and Hubbard's conjecture (see \cite{QY} and \cite{KvS}). For rational maps, the Julia sets may exhibit more complex topological structures. Pilgrim and Tan proved that if the Julia set of a hyperbolic (more generally, geometrically finite) rational map is disconnected, then, with the possible exception of finitely many periodic components and their countable collection of preimages, every Julia component is either a single point or a Jordan curve \cite{PT}.

For the general rational maps, it is not easy to describe the topological structures of the corresponding Julia sets. However, for some special families of rational maps, the topological properties of the Julia sets can be studied well. Recently, Devaney considered a singular perturbation of the complex polynomials by adding a pole and studied the dynamics of the rational maps with the following form
$$
F_{\lambda}(z)=P(z)+\frac{\lambda}{(z-a)^d},
$$
where $P(z)$ is a polynomial with degree $n\geq 2$ whose dynamics are completely understood, $a\in \C$, $d\geq 1$ and $\lambda\in\C^*$ \cite{Dev}. When $P(z)=z^n$ with $n\geq 2$, $a=0$, $d\geq 2$ and $\lambda\in\C^*$, the family of rational maps $F_\lambda$ is commonly called the \emph{McMullen maps}, which has been studied extensively by Devaney and his collaborators in a series of articles (see \cite{BDLSS, Dev, DLU, DR}). Specifically, it is proved in \cite{DLU} that if the orbits of the critical points of $F_{\lambda}$ are all attracted to $\infty$, then the Julia set of $F_{\lambda}$ is either a Cantor set, a Sierpi\'nski curve, or a Cantor set of circles. In particular, the Julia set of $F_\lambda$ is a Cantor set of circles if $1/n+1/d<1$ and $\lambda\neq 0$ is small. If the orbits of the free critical points of $F_{\lambda}$ are bounded, then $F_\lambda$ has no Herman rings \cite{XQ} and actually, the corresponding Julia set is connected \cite{DR}. Since the McMullen family behaves extremely rich dynamics, this family has also been studied in \cite{QWY} and \cite{St1}. When $P(z)=z^n+b$ with $n\geq 2$, $a=0$, $d\geq 2$ and $\lambda$, $b\in\C^*$, the family of maps $F_\lambda$ is called the \emph{generalized McMullen maps}, which also attracts many people's interest. Some additional dynamical phenomenon happens for this family since the parameter space becomes $\C^*\times\C^*$, which is two-dimensional.  For a comprehensive study on the generalized McMullen map, see \cite{BDGMR, GG, KD, XQY} and the references therein. There exist also some other special families of rational maps which were studied well. For example, see \cite{FY}, \cite{JS} and \cite{St2}.

Note that for McMullen maps and generalized McMullen maps, the point at infinity is always a super-attracting fixed point and the origin is the unique pole. If the parameter is close to the origin, then each of these maps can be seen as a perturbation of the simple polynomial $P_n(z)=z^n$. Recently, the singular perturbation with multiple poles has been considered. For example, Garijo, Marotta and Russell studied the singular perturbations of the form $G_{\lambda}(z)=z^2-1+\lambda/((z^{d_0}(z + 1)^{d_1})$ in \cite{GMR} and focused on the topological characteristics of the Julia and Fatou sets of $G_{\lambda}$ that arise when the parameter $\lambda$ becomes nonzero.

In this article, we consider the following family of rational maps
\begin{equation}\label{our-family}
f_\lambda(z)= z^n+\frac{\lambda^2}{z^n-\lambda},
\end{equation}
where $n\geq 3$ and $\lambda\in\C^*$. This family can be also seen as a perturbation of the simple polynomial $P_n$ if $\lambda$ is small. We would like to mention that the perturbation here is essentially different from Devaney's family $F_\lambda$ (including McMullen maps and the generalized McMullen maps) and Garijo-Marotta-Russell's family $G_\lambda$ since the map $f_\lambda$ has multiple poles and the origin is no longer a pole.

It is easy to see that $f_\lambda$ has a super-attracting fixed point at $\infty$. Since the degree is $2n$, the map $f_\lambda$ has $4n-2$ critical points (counted with multiplicity). Note that the local degree of $\infty$ is $n$ and the origin is another critical point with local degree $2n$ (see \eqref{equ-family-other}) whose critical value is $v_0=f_\lambda(0)=-\lambda$.  Hence, this leaves $n$ more critical points. We call the remaining $n$ critical points and $0$ the \emph{free} critical points. In \S\ref{subsec-dyn-sym} we will show that these $n$ free critical points (except $0$) have a common critical value $v_1=3\lambda$ (see \eqref{free-cv}).  We call these two critical values $v_0$ and $v_1$ the \emph{free} critical values. The dynamics of $f_\lambda$ is determined by the orbits of these two free critical values. In this article, we will give a quite complete description of the Julia sets of the family $f_\lambda$ for arbitrary parameter $\lambda\in\C^*$.

\subsection{Statement of the results}

In the rest of this article, we  use $\MA(\infty)$ to denote the super-attracting basin of $f_\lambda$ containing $\infty$. Recall that $v_0=-\lambda$ and $v_1=3\lambda$ are two free critical values of $f_\lambda$.

\begin{thm}\label{main-thm}
For $n\geq 3$ and $\lambda\in\C^*$, the Julia set $J(f_\lambda)$ of $f_\lambda$ is one of the following cases:
\begin{itemize}
\item[(1)] If $v_0\in\MA(\infty)$, then $J(f_\lambda)$ is a Cantor set;
\item[(2)] If $v_0\not\in\MA(\infty)$ and $v_1\in\MA(\infty)$, then $J(f_\lambda)$ is connected;
\item[(3)] If $v_0$, $v_1\notin \MA(\infty)$, then there are two possibilities:
  \begin{itemize}
  \item[(3a)] If each Fatou component contains at most one free critical value, then $J(f_\lambda)$ is connected;
  \item[(3b)] If $v_0$ and $v_1$ lie in the same Fatou component, then $J(f_\lambda)$ is the union of countably many Jordan curves and uncountably many points and hence disconnected.
  \end{itemize}
\end{itemize}
\end{thm}

We will also give several typical examples to show that all the types of the Julia sets stated in Theorem \ref{main-thm} actually happen. The parameters
that correspond to the examples are chosen from the parameter plane of $f_\lambda$ with $n=3$. See Figure \ref{Fig_parameter-space}.

\begin{figure}[!htpb]
  \setlength{\unitlength}{1mm}
  \centering
  \includegraphics[height=80mm]{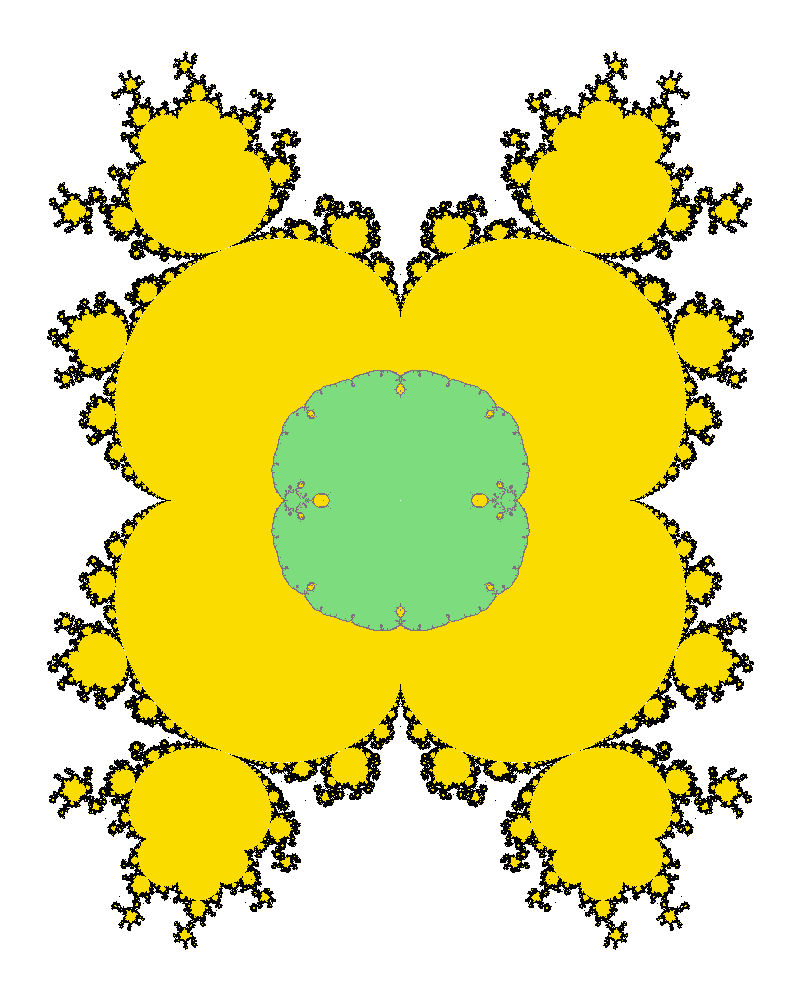}\quad
  \includegraphics[height=80mm]{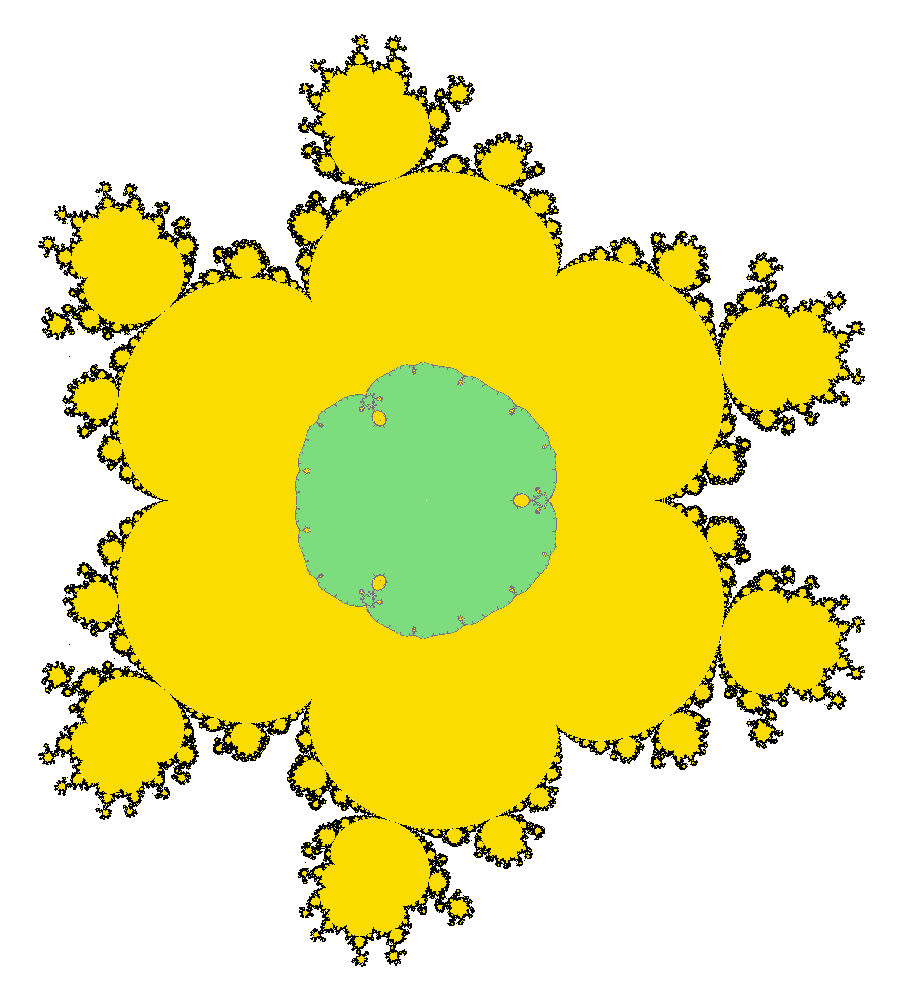}
  \caption{The parameter planes (i.e. $\lambda$-plane) of $f_\lambda$, where $n=3$ and $n=4$ (from left to right). In both pictures, the yellow and green parts denote the parameters $\lambda$ such that the free critical values $v_0$ and $v_1$, respectively, are not attracted by $\infty$. It can be seen from these pictures that the green parts are compactly contained in the yellow parts (some of the central yellow places are covered by the green parts).}
  \label{Fig_parameter-space}
\end{figure}

The Julia sets in Theorem \ref{main-thm} (3b) are called \textit{Cantor bubbles} (see the picture on the right of Figure \ref{Fig_Julia-thm3}). This kind of Julia sets has also been found by Devaney and Marotta in \cite{DM} for the rational maps $z\mapsto z^n+c/(z-a)^d$ when $|a|\neq 0,1$, $c$ is sufficiently small, and $1/n+1/d<1$.

A subset of the Riemann sphere $\EC$ is called a \textit{Cantor set of circles} (or \textit{Cantor circles} in short) if it consists of uncountably many closed Jordan curves which is homeomorphic to $\mathcal{C}\times \mathbb{S}^1$, where $\mathcal{C}$ is the Cantor middle third set and $\mathbb{S}^1$ is the unit circle. By definition, a \textit{Sierpi\'nski curve} is a planar set homeomorphic to the well-known Sierpi\'nski carpet fractal. From Whyburn \cite{Why}, it is known that any planar set which is compact, connected, locally connected, nowhere dense, and has the property that any two complementary domains are bounded by disjoint simple closed curves is homeomorphic to the Sierpi\'nski curve. It is known that the Cantor circles Julia sets and Sierpi\'nski curves Julia sets can appear in McMullen family and the generalized McMullen family (see \cite{DLU} and \cite{XQY}). However, for the family $f_\lambda$, it is proved that these two kind of Julia sets are not exist.

\begin{thm}\label{thm-no-type}
For any $n\geq 3$ and $\lambda\in\C^*$, the Julia set of $f_\lambda$ can never be a Cantor set of circles or a Sierpi\'{n}ski curve. Moreover, $f_\lambda$ has no Herman rings.
\end{thm}

One can refer \cite{QYY} for the comprehensive study on the rational maps whose Julia sets are Cantor circles. The first example of the Sierpi\'{n}ski curve as the Julia set of a rational map was given in \cite[Appendix F]{Mi1}. For more rational maps whose Julia sets are Sierpi\'{n}ski curves, see \cite{DFGJ}. For the study of non-existence of Herman rings for rational families, see \cite{XQ}, \cite{Ya} and the references therein.

\subsection{Organization of the article}

The article is organized as follows:

In \S \ref{Prelim}, we introduce the family $f_\lambda$ and present some basic properties of $f_\lambda$. Some useful lemmas which are necessary in the proofs of our theorems are also prepared.

In \S \ref{Escape}, we describe the Julia set of $f_\lambda$ for the case that the free critical value $v_0$ is attracted by $\infty$ and show that under this assumption, $v_1$ is also attracted by $\infty$ and the Julia set of $f_\lambda$ is a Cantor set.

In \S \ref{SEscape}, we discuss the case that the super-attracting fixed point $\infty$ attracts exactly one free critical value $v_1$ and prove that $J(f_\lambda)$ is connected.

In \S \ref{main-thmape}, we deal with the case that neither $v_0$ nor $v_1$ are attracted by the super-attracting basin of $\infty$ and show $J(f_\lambda)$ is either connected or a set of Cantor bubbles.

At the end of \S \ref{Escape} to \S \ref{main-thmape}, we also give typical examples to show that the Julia sets appeared in Theorem \ref{main-thm} actually happen.

In \S \ref{sec-impo}, we prove Theorem \ref{thm-no-type} by constructing several polynomial-like mappings.

In the last section, we make some comments on $f_\lambda$ with $n=2$. We conjecture that Theorems \ref{main-thm} and \ref{thm-no-type} hold also in this case. However, we cannot give a proof here. Comparing Figures \ref{Fig_parameter-space} and \ref{Fig_parameter-n-2}, there is a slight difference between them: the green part is compactly contained in the yellow part if $n\geq 3$ while these two parts have non-empty intersection on their boundaries if $n=2$. This is the essential obstruction that we cannot use the techniques in this article to deal with the case $n=2$.

\vskip0.2cm
\noindent\textbf{Acknowledgements.} The first author is supported by the NSFC (Nos. \,11301165, 11371126, 11571099) and the program of CSC (2015/2016). He also wants to acknowledge the Department of Mathematics, Graduate School of the City University of New York for its hospitality during his visit in 2015/2016. The second author is supported by the NSFC (No.\,11401298) and the NSF of Jiangsu Province (No.\,BK20140587). We would like to thank the referee for careful reading and useful suggestions.

\section{Preliminaries}\label{Prelim}

In this section, we prepare some preliminary results. We first give the symmetric distribution of the critical points and the symmetric dynamical behaviors of $f_\lambda$. Then we consider the topological properties of the immediate super-attracting basin of $\infty$. In the rest of this article, we always assume that $n\geq 3$ is an integer if there is no other special instruction.

\subsection{Dynamical Symmetries}\label{subsec-dyn-sym}

As pointed out in the introduction, the rational map
\begin{equation}\label{equ-family-other}
f_\lambda(z)= z^n+\frac{\lambda^2}{z^n-\lambda}=\frac{z^{2n}-\lambda z^n+\lambda^2}{z^n-\lambda}
\end{equation}
has a super-attracting fixed point $\infty$, which is also a critical point of $f_\lambda$ with multiplicity $n-1$. A direct calculation shows that
\begin{equation}\label{equ-f-dev}
f_\lambda'(z)=nz^{2n-1}\cdot\frac{z^n-2\lambda }{(z^n-\lambda )^2}.
\end{equation}
It is easy to see that the origin is another critical point of $f_\lambda(z)$ with multiplicity $2n-1$. The rest $n$ critical points of $f_\lambda$ have the form
\begin{equation*}
\{c_k:=\omega^{k-1}\sqrt[n]{2\lambda }:1\leq k\leq n\}, \text{~where~}\omega=e^{2\pi\ii/n}.
\end{equation*}
However, except $\infty$, there are only two critical values for these critical points. They are
\begin{equation}\label{free-cv}
v_0:=f_\lambda(0)=-\lambda  \text{ and } v_1:=f_\lambda(c_{k})=3\lambda \text{ for } 1\leq k\leq n.
\end{equation}
In this article, we call $0$, $c_k$ the \emph{free} critical points, and $v_0,v_1$ the \emph{free} critical values of $f_\lambda$. The dynamics of $f_\lambda$ is determined by the orbits of these two free critical values. Since the local degree of $f_\lambda$ at the origin is $2n$ and the local degree of $f_\lambda$ is two at every free critical point $c_k$, we have \begin{equation}\label{equ-v-c-inv}
f^{-1}_\lambda(v_0)=\{0\} \text{ and } f^{-1}_\lambda(v_1)=\{c_k:1\leq k\leq n\}.
\end{equation}

Recall that $B_\infty$ is the immediate super-attracting basin of $\infty$. Let $U$ be a subset of $\EC$ and $\alpha\in\C$. We denote $\alpha U:=\{\alpha z:z\in U\}$. The proof of the following lemma is straightforward.

\begin{lema}\label{lema1}
Let $\omega$ be a complex number satisfying $\omega^n=1$ and suppose that $U$ is a Fatou component of $f_\lambda$. Then

$(1)$ $f_\lambda(\omega z)=f_\lambda(z)$ and $\omega U$ is also a Fatou component of $f_\lambda$.

$(2)$ The basin $B_\infty$ has \emph{$n$-fold symmetry}, i.e. $z\in B_\infty$ if and only if $\omega z\in B_\infty$.

$(3)$ Let $U$ be a Fatou component of $f_\lambda$ which is different from $B_\infty$. Then either $U$ has $n$-fold symmetry and surrounds the origin, or $\omega U$, $\omega^2 U$, $\cdots$, $\omega^n U=U$ are pairwise disjoint, where $\omega=e^{2\pi\ii/n}$.
\end{lema}

In this article, we need to prove that some domains are simply connected and the following formula is very useful.

\begin{lema}[{Riemann-Hurwitz's formula, \cite[\S 5.4, pp.\,85-89]{Bea}}]\label{lema-Riem-Hur}
Let $f$ be a rational map defined from $\EC$ to itself. Assume that

\textup{(1)} $V$ is a domain in $\EC$ with finitely many boundary components;

\textup{(2)} $U$ is a component of $f^{-1}(V)$; and

\textup{(3)} there are no critical values of $f$ on $\partial V$.

Then there exists an integer $d\geq 1$ such that $f$ is a branched covering map from $U$ onto $V$ with degree $d$ and
\begin{equation*}
\chi(U)+\delta_f(U)=d\cdot\chi(V),
\end{equation*}
where $\chi(\cdot)$ denotes the Euler characteristic and $\delta_f(U)$ denotes the total number of the critical points of $f$ in $U$ (counted with multiplicity).
\end{lema}

\begin{rmk}
Let $D$ be a domain in $\EC$. Then $\chi(D)=2$ if and only if $D$ is the Riemann sphere $\EC$; $\chi(D)=1$ if and only if $D$ is simply connected; and $\chi(D)=0$ if and only if $D$ is doubly connected (i.e. an annulus).
\end{rmk}

\subsection{The topological structure of $B_\infty$}

A simply connected domain in $\EC$ is called a \emph{Jordan domain} if its boundary is a Jordan curve.

\begin{lema}\label{connectivity}
Let $U\subset\EC$ be a simply connected domain which contains exactly one free critical value $v_0$. Then the preimage $f_\lambda^{-1}(U)$ is a simply connected domain containing $0$ on which the degree of the restriction of $f_\lambda$ is $2n$.
\end{lema}

\begin{proof}
Note that the simply connected domain $U$ contains exactly one critical value $v_0=-\lambda$ and $f_\lambda^{-1}(v_0)=\{0\}$. This means that $f_\lambda^{-1}(U)$ is a connected set containing $0$. By Riemann-Hurwitz's formula (Lemma \ref{lema-Riem-Hur}), it follows that $f_\lambda^{-1}(U)$ is also simply connected on which the degree of the restriction of $f_\lambda$ is $2n$.
\end{proof}

\begin{prop}\label{inBandT}
If $B_\infty$ contains at least one free critical value, then $B_\infty$ is completely invariant and $J(f_\lambda)=\partial B_\infty$.
\end{prop}

\begin{proof}
Suppose that $v_0\in B_\infty$. Since $f^{-1}_\lambda(v_0)=\{0\}$ and $f_\lambda(B_\infty)=B_\infty$, we have $0\in B_\infty$ and $B_\infty$ is completely invariant. If $v_1\in B_\infty$, then $B_\infty$ contains at least one free critical point $c_k$ by \eqref{equ-v-c-inv}. Since $B_\infty$ has $n$-fold symmetry, we obtain that $f^{-1}(v_1)\subset B_\infty$, which implies that $B_\infty$ is completely invariant. The assertion $J(f_\lambda)=\partial B_\infty$ follows by \cite[Corollary 4.12]{Mi2}.
\end{proof}

For the connectivity of the Julia sets of rational maps, the following criterion was established in \cite{XY}.

\begin{lema}[{\cite[Lemma 2.9]{XY}}]\label{CFPCS}
Suppose that $f$ is a rational function which has no Herman rings and each Fatou component contains at most one critical value. Then the Julia set of $f$ is connected.
\end{lema}

We remark that Peherstorfer and Stroh proved a similar result as Lemma \ref{CFPCS} in \cite[Theorem 4.2]{PS}, where they required that each Fatou component contains at most one critical \emph{point} (counted without multiplicity).

\section{Both free critical values are escaped}\label{Escape}

In this section, we consider the case where the free critical value $v_0$ is attracted by $\infty$. However, we will prove that $v_1$ lies in the super-attracting basin of $\infty$ if $v_0$ does. According to Sullivan's classification theorem, the Fatou set $F(f_\lambda)$ of $f_\lambda$ is equal to $\MA(\infty)$ and the Julia set is $J(f_\lambda)=\EC\setminus\MA(\infty)$. For $a\in\C$ and $r>0$, we use $\D(a,r)=\{z\in\C:|z-a|<r\}$ to denote the open disk centered at $a$ with radius $r$.

\begin{lema}\label{lema-domain-a-1}
For any $0<\kappa<1$, let $\lambda$ be the parameter satisfying
\begin{equation}\label{equ-a-domain-1}
0<|\lambda|< \frac{1}{1+\kappa}\left(\frac{2\kappa}{(1+\kappa)(\sqrt{\kappa^2+4\kappa}+\kappa)}\right)^{\frac{1}{n-1}}.
\end{equation}
Denote $D_\lambda:=\D(-\lambda ,\kappa|\lambda|)$. Then we have

$(1)$ $f_\lambda$ maps the closed disk $\overline{D}_\lambda$ into its interior and $\overline{D}_\lambda$ is contained in a geometrically attracting basin $B_0$ of $f_\lambda$;

$(2)$ The preimage $f_\lambda^{-1}(D_\lambda)$ is a Jordan domain containing $0$ on which the degree of the restriction of $f_\lambda$ is $2n$. In particular, the basin $B_0$ is completely invariant.
\end{lema}

\begin{proof}
(1) If $\lambda$ satisfies \eqref{equ-a-domain-1}, then we have
\begin{equation*}
(1+\kappa)^n|\lambda|^{n-1}< \frac{2\kappa}{\sqrt{\kappa^2+4\kappa}+\kappa}<1,
\end{equation*}
which means that $(1+\kappa)^n|\lambda|^n<|\lambda|$. If $z\in \overline{D}_\lambda=\overline{\D}(-\lambda ,\kappa|\lambda|)$, then we have
\begin{equation*}
\begin{split}
|f_\lambda(z)-(-\lambda )|&~=\left|\frac{z^{2n}}{z^n-\lambda }\right|\leq\frac{(1+\kappa)^{2n}|\lambda|^{2n}}{|\lambda|-(1+\kappa)^n|\lambda|^n}
=|\lambda|\cdot\frac{(1+\kappa)^{2n}|\lambda|^{2n-2}}{1-(1+\kappa)^n|\lambda|^{n-1}}\\
&~< |\lambda|\cdot\left(\frac{2\kappa}{\sqrt{\kappa^2+4\kappa}+\kappa}\right)^2\cdot\left(1-\frac{2\kappa}{\sqrt{\kappa^2+4\kappa}+\kappa}\right)^{-1}=\kappa|\lambda|.
\end{split}
\end{equation*}
This means that $f_\lambda$ maps the closed disk $\overline{D}_\lambda$ into its interior. Therefore, $\overline{D}_\lambda$ is contained in a fixed Fatou component $B_0$ of $f_\lambda$ and the orbit of $v_0=-\lambda $ is contained in $D_\lambda$.

Since the critical point $0$ and the critical value $v_1=3\lambda$ are both disjoint with $\overline{D}_\lambda$, it means that $\overline{D}_\lambda$ does not contain any critical points. By Schwarz's Lemma, $B_0$ is a geometrically attracting basin which contains an attracting fixed point in $D_\lambda$ with multiplier $\rho$ satisfying $0<|\rho|<1$.

(2) Since $D_\lambda$ is a Jordan domain containing exactly one free critical value $v_0$. The first assertion holds by Lemma \ref{connectivity}. Moreover, the attracting basin containing $D_\lambda$ is completely invariant.
\end{proof}

\begin{lema}\label{lema-domain-a-2}
For $\kappa=1/5$, let $\lambda$ be the parameter satisfying
\begin{equation}\label{equ-a-domain-2}
|\lambda|\geq \frac{1}{1+\kappa}\left(\frac{2\kappa}{(1+\kappa)(\sqrt{\kappa^2+4\kappa}+\kappa)}\right)^{\frac{1}{n-1}}.
\end{equation}
Then $f_\lambda$ maps the closed disk $\EC\setminus\D(0,3|\lambda|)$ into its interior. In particular, the orbit of the free critical value $v_1=3\lambda$ is contained in the immediate attracting basin of $\infty$.
\end{lema}

\begin{proof}
For simplicity, we denote $A:=3|\lambda|$. If $\lambda$ satisfies \eqref{equ-a-domain-2}, since $\kappa=1/5$ and $n\geq 3$, we have
\begin{equation*}
\begin{split}
A^{n-1}&~\geq \left(\frac{3}{1+\kappa}\right)^{n-1}\frac{2\kappa}{(1+\kappa)(\sqrt{\kappa^2+4\kappa}+\kappa)}=\left(\frac{5}{2}\right)^{n-1}\cdot\frac{\sqrt{21}-1}{12}\\
&~> \frac{25}{4}\cdot\frac{1}{4}>\frac{2+\sqrt{2}}{3}>1.
\end{split}
\end{equation*}
Therefore, $A^n>A>|\lambda|$ and $9A^{2n-2}-3A^{n-1}-1>9A^{n-1}-3$.  If $|z|\geq A$, then
\begin{equation*}
|f_\lambda(z)|=\left|z^n+\frac{\lambda^2}{z^n-\lambda }\right|\geq A^n-\frac{A^2/9}{A^n-A/3}=A\cdot\frac{9A^{2n-2}-3A^{n-1}-1}{9A^{n-1}-3}>A.
\end{equation*}
Thus $f_\lambda$ maps the closed disk $\EC\setminus\D(0,3|\lambda|)$ into its interior. The proof is complete.
\end{proof}

Recall that $B_\infty$ is the immediate super-attracting basin of $\infty$.

\begin{cor}\label{cor-v-0-v-c}
If $v_0\in\MA(\infty)$, then $v_1\in B_\infty\subset \MA(\infty)$.
\end{cor}

\begin{proof}
If $v_0=-\lambda$ is attracted by the super-attracting fixed point located at $\infty$, then $\lambda$ should satisfy \eqref{equ-a-domain-2} for any $0<\kappa<1$ by Lemma \ref{lema-domain-a-1}. Let us set $\kappa=1/5$. By Lemma \ref{lema-domain-a-2}, we know that $v_1=3\lambda\in B_\infty\subset\MA(\infty)$.
\end{proof}

In order to prove Theorem \ref{main-thm}, we need the following lemma.

\begin{lema}[{\cite[Theorem 9.8.1]{Bea}}]\label{Beardon}
Let $f$ be a rational map with degree at least two. If all of the critical points of $f$ lie in the immediate attracting basin of a (super)attracting fixed point of $f$, then the Julia set of $f$ is a Cantor set.
\end{lema}

\begin{proof}[{Proof of Theorem \ref{main-thm} (1)}]
If $v_0$ is attracted by the super-attracting fixed point located at $\infty$, then the critical value $v_1$ lies in the immediate super-attracting basin $B_\infty$ by Corollary \ref{cor-v-0-v-c}. According to Proposition \ref{inBandT}, $B_\infty$ is completely invariant and contains all critical points. This means that the Julia set of $f_\lambda$ is a Cantor set by Lemma \ref{Beardon}.
\end{proof}

\begin{exam}
For each $n\geq 3$, let $\lambda=-\sqrt[n-1]{-1}$. Then it is easy to check that $v_0$ is a pole of $f_\lambda$, i.e. $v_0\in\MA(\infty)$. Therefore, by Theorem \ref{main-thm} (1), $J(f_\lambda)$ is a Cantor set.
\end{exam}

\section{Only one free critical value is escaped}\label{SEscape}

In this section, we consider the case where the attracting basin of $\infty$ attracts exactly one free critical orbit of $f_\lambda$. By Corollary \ref{cor-v-0-v-c}, there is only one possibility: $v_1\in \MA(\infty)$ and $v_0\notin \MA(\infty)$. In order to prove Theorem \ref{main-thm} (2), we need the polynomial-like mapping theory introduced by Douady and Hubbard in \cite{DH}.

\begin{defi}
A triple $(U, V, f)$ is called a \emph{polynomial-like mapping} of degree $d\geq 2$ if $U$ and $V$ are simply connected plane domains such that $\overline{U}\subset V$, and $f: U\rightarrow V$ is a holomorphic proper mapping of degree $d$. The \emph{filled Julia set} $K(f)$ of a polynomial-like mapping $f$ is defined as
$$
K(f)=\{z\in U: f^{\circ k}(z)\in U \text{~for all~}  k\geq 0\}.
$$
The \emph{Julia set} of the polynomial-like mapping $f$ is defined as $J(f)=\partial K(f)$.
\end{defi}

Two polynomial-like mappings $(U_{1}, V_{1}, f_1)$ and $(U_{2},V_{2}, f_{2})$ of degree $d\geq 2$ are said to be \emph{hybrid equivalent} if there exists a quasi-conformal homeomorphism $h$ defined from a neighborhood of $K(f_1)$ onto that of $K(f_2)$, which conjugates $f_1$ to $f_2$ and whose complex dilatation vanishes on
$K(f_1)$. The following theorem was proved by Douady and Hubbard in \cite[Theorem 1, p.\,296]{DH}.

\begin{thm}[{The Straightening Theorem}]\label{straightening}
Every polynomial-like mapping $(U, V, f)$ of degree $d\geq 2$ is hybrid equivalent to a polynomial with the same degree.
\end{thm}

By applying Theorem \ref{straightening} and Fatou's theorem \cite[Theorem 4.1, p.\,66]{CG}, the following result was proved in \cite{XY}.

\begin{lema}[{\cite[Corollary 4.2]{XY}}]\label{cor-straight-thm}
Suppose $(U,V,f)$ is a polynomial-like mapping of degree $d\geq 2$. Then the Julia set of $f$ is connected if and only if all critical points of $f$ are contained in the filled Julia set of $f$.
\end{lema}

The following lemma is very useful when one wants to prove the non-existence of Herman rings in the holomorphic family.

\begin{lema}[{\cite[Main Theorem]{Ya}}]\label{lema-no-Herman}
Any rational map having at most one critical orbit in its Julia set has no Herman rings.
\end{lema}

\begin{proof}[{Proof of Theorem \ref{main-thm} (2)}]
The arguments will be divided into two cases: the first one is $v_1\in B_\infty$ and the second one is $v_1\in \MA(\infty)\setminus B_\infty$.

(i) Suppose that $v_1\in B_\infty$. Then $B_\infty$ is completely invariant and $J(f_\lambda)=\partial B_\infty$ by Proposition \ref{inBandT}. We will choose some suitable domains to construct a polynomial-like mapping and prove that its Julia set is connected and quasi-conformally conjugate to the Julia set of $f_\lambda$.

Since $\infty$ is a super-attracting fixed point, we can choose a small simply connected neighborhood $\Omega_0$ of $\infty$ such that $v_1\not\in \overline{\Omega}_0$, $f_\lambda(\overline{\Omega}_0)\subset \Omega_0$ and $\partial \Omega_0$ is a Jordan curve which is disjoint from the forward orbit of $v_1$ (note that the forward orbit of $v_1$ is discrete except the unique possible accumulation point at $\infty$). For $m\geq 0$, let $\Omega_m$ be the connected component of $f_\lambda^{-m}(\Omega_0)$
containing $\Omega_0$. Then we have $\Omega_0\subset \Omega_1\subset \Omega_2\subset\cdots\subset \Omega_m\subset\cdots$
and $B_\infty=\bigcup_{m\geq 0}\Omega_m$. Since $v_1\in B_\infty$, there must exist $m_0\geq 1$ such that $v_1\in \Omega_{m_0}\setminus \overline{\Omega}_{m_0-1}$. By Lemma \ref{lema-Riem-Hur}, it follows that both $\Omega_{m_0}$ and $V:=\EC\setminus\overline{\Omega}_{m_0}$ are simply connected. Moreover, since $\partial \Omega_0$ is a Jordan curve which is disjoint from the critical orbit of $v_1$, this means that $\Omega_{m_0}$ and $V$ are both Jordan domains. Note that $v_0\in V$ and $v_1\notin V$. It follows that $f_\lambda^{-1}(V)\subset V$ is a Jordan domain by Lemma \ref{connectivity}.

Let $U=f_\lambda^{-1}(V)$. We obtain a polynomial-like mapping $(U, V, f_\lambda)$ with degree $2n$. Since $v_0\not\in\MA(\infty)$, it means that the unique critical orbit of $(U,V,f_\lambda)$ is contained in $U$. Therefore, the Julia set of $(U,V,f_\lambda)$ is connected by Lemma \ref{cor-straight-thm}. Since $\EC\setminus V\subset\MA(\infty)$, we know that the Julia set of the polynomial-like mapping $(U,V,f_\lambda)$ is homeomorphic to that of the rational map $f_\lambda$. Therefore, the Julia set of the rational map $f_\lambda$ is connected.

(ii) If $v_1\in \MA(\infty)\setminus B_\infty$, then each Fatou component of $f_\lambda$ contains at most one critical value. Note that $f_\lambda$ contains at most one critical orbit in its Julia set. It follows that $f_\lambda$ has no Herman rings by Lemma \ref{lema-no-Herman}. According to Lemma \ref{CFPCS}, this means that $J(f_\lambda)$ is connected. The proof is complete.
\end{proof}

We now give specific examples to show that the Julia sets discussed in the proof of Theorem \ref{main-thm} (2) happen indeed.

\begin{exam}\label{exam-1}
(i) Let $n=3$ and $\lambda=e^{\frac{\pi\ii}{3}}$. Then $v_0\not\in\MA(\infty)$ and $v_1\in B_\infty\subset\MA(\infty)$;

(ii) Let $n=3$ and $\lambda=\sqrt{3}/9$. Then $v_0\not\in\MA(\infty)$ and $v_1\in \MA(\infty)\setminus B_\infty$.
\end{exam}

See Figure \ref{Fig_Julia-thm2} for the Julia sets of Example \ref{exam-1}.

\begin{figure}[!htpb]
  \setlength{\unitlength}{1mm}
  \centering
  \includegraphics[width=70mm]{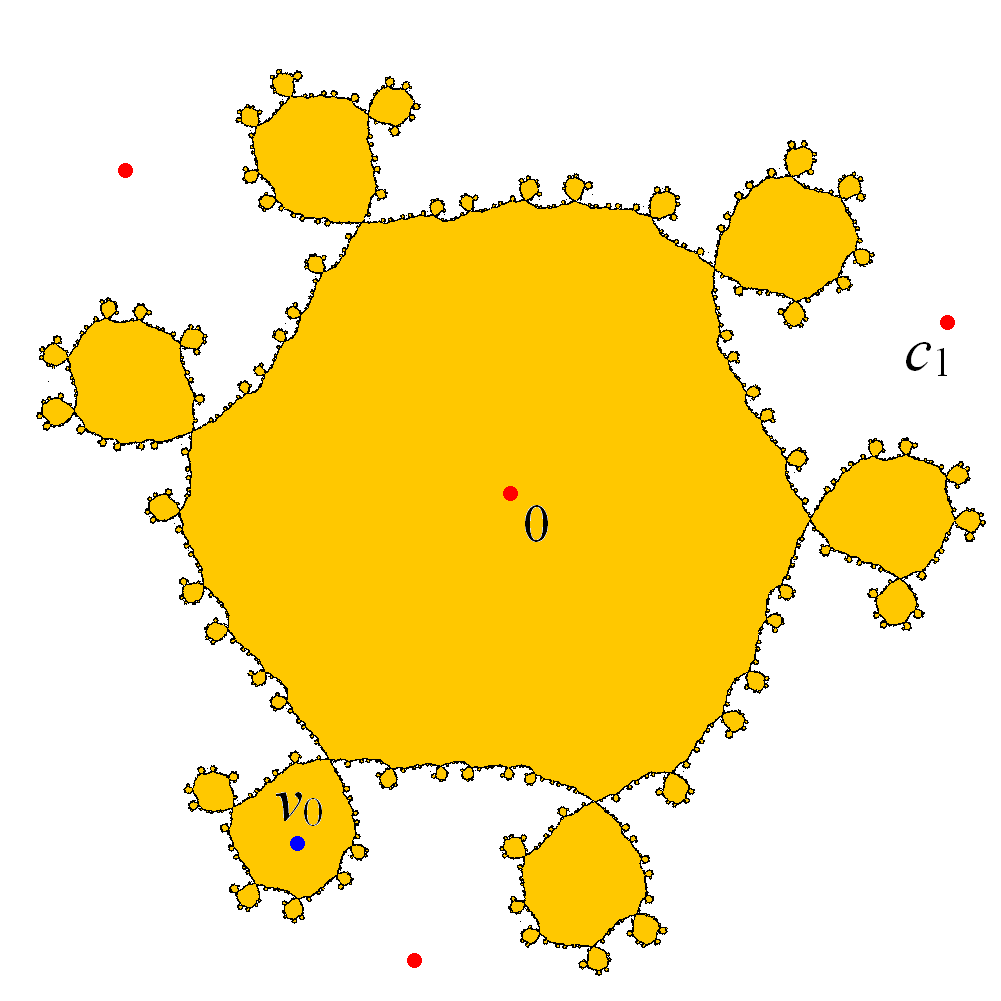}
  \includegraphics[width=70mm]{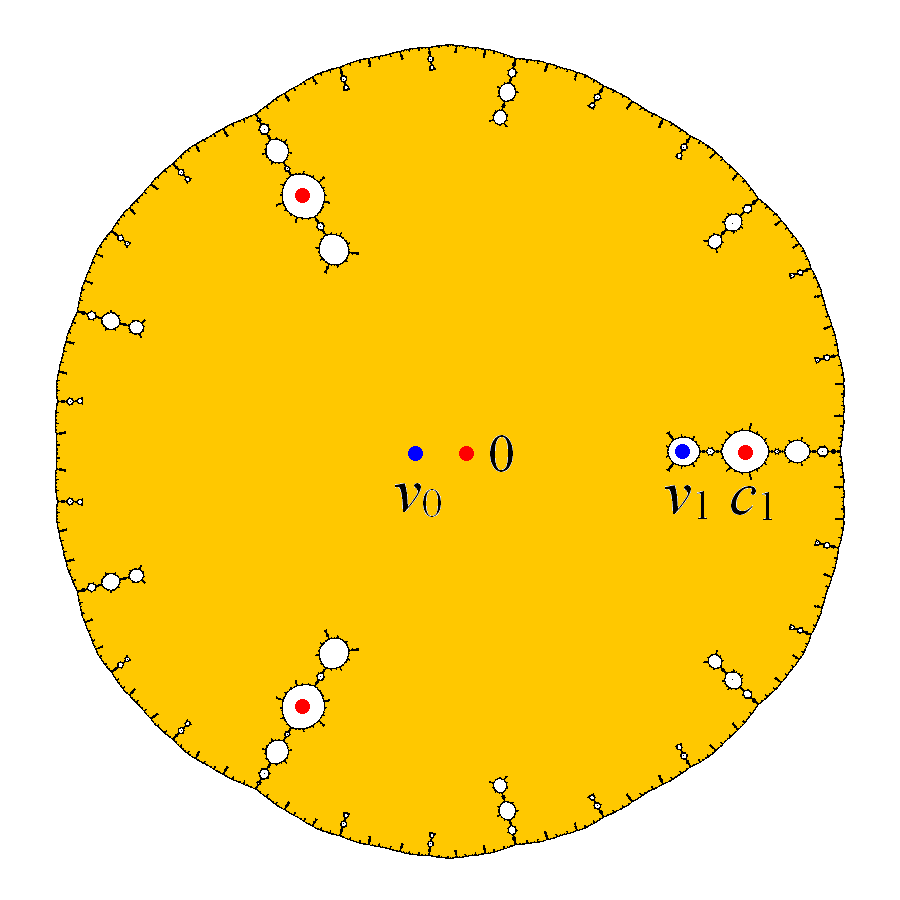}
  \caption{The Julia sets of $f_\lambda$ with different parameters $\lambda_1=e^{\frac{\pi\ii}{3}}$ and $\lambda_2=\sqrt{3}/9$ (from left to right), where $n=3$. The parameter $\lambda_1$ is chosen such that the free critical value $v_1$ lies in the immediate attracting basin of $\infty$ while $\lambda_2$ is chosen such that $v_1$ lies in the attracting basin but not in the immediate attracting basin of $\infty$. These two Julia sets correspond to the two cases that appeared in the proof of Theorem \ref{main-thm} (2). The free critical points and values are marked by red and blue dots respectively.}
  \label{Fig_Julia-thm2}
\end{figure}

\begin{proof}
(i) If $\lambda=e^{\frac{\pi\ii}{3}}$, then $(-\lambda)^3=1$. By \eqref{equ-family-other}, we have
\begin{equation*}
f_\lambda(v_0)=\frac{(-\lambda)^6-\lambda (-\lambda)^3+\lambda^2}{(-\lambda)^3-\lambda}=\frac{1-\lambda+\lambda^2}{1-\lambda}=0.
\end{equation*}
This means that $0\mapsto v_0\mapsto 0$ forms a super-attracting periodic orbit with period two and hence $v_0\not\in\MA(\infty)$. If $n=3$ and $\kappa=1/5$, by \eqref{equ-a-domain-2}, we have
\begin{equation*}
\frac{1}{1+\kappa}\Big(\frac{2\kappa}{(1+\kappa)(\sqrt{\kappa^2+4\kappa}+\kappa)}\Big)^{\frac{1}{n-1}}=\frac{5}{6}\Big(\frac{\sqrt{21}-1}{12}\Big)^{\frac{1}{2}}
<1=|e^{\frac{\pi\ii}{3}}|.
\end{equation*}
By Lemma \ref{lema-domain-a-2}, we have $v_1\in B_\infty\subset\MA(\infty)$.

(ii) If $n=3$ and $\kappa=1/5$, we have\footnote{The number $\sqrt{6}/9$ in this inequality will be used in the construction of an example in next section.}
\begin{equation}\label{equ-cal-1}
\frac{1}{1+\kappa}\Big(\frac{2\kappa}{(1+\kappa)(\sqrt{\kappa^2+4\kappa}+\kappa)}\Big)^{\frac{1}{n-1}}>\frac{5}{6}\Big(\frac{4-1}{12}\Big)^{\frac{1}{2}}
=\frac{5}{12}>\frac{\sqrt{6}}{9}>\frac{\sqrt{3}}{9}.
\end{equation}
By Lemma \ref{lema-domain-a-1}, we have $v_0\not\in\MA(\infty)$ if $\lambda=\sqrt{3}/9$. Note that $(3\lambda)^3-\lambda=0$. This means that $f_\lambda(v_1)=\infty$ and hence $v_1\in \MA(\infty)$. Therefore, we only need to prove that $v_1\not\in B_\infty$.

Since $\lambda=\sqrt{3}/9$ is real, we consider the restriction of $f_\lambda$ on the real axis $\R$. According to \eqref{equ-family-other}, $f_\lambda$ is continuous in the interval $(\sqrt[3]{\lambda},+\infty)=(v_1,+\infty)$. Recall that $c_1=\sqrt[3]{2\lambda}$ is a free critical point of $f_\lambda$. Note that $f_\lambda(c_1)=v_1=\sqrt[3]{\lambda}<c_1$ and there exists a sufficiently large $M>1$ such that $f_\lambda(M)>M$ and $M\in B_\infty$. There must exist a real number $x_1\in(c_1, M)$ such that $f_\lambda(x_1)=x_1$. Since all the critical values of $f_\lambda$ lie in (super) attracting basins, it follows that $x_1$ is a repelling fixed point and hence contained in the Julia set of $f_\lambda$.

Suppose that $v_1\in B_\infty$. Then $c_1\in B_\infty$ by Proposition \ref{inBandT}. There exists a smooth curve $\gamma:[0,1]\to B_\infty$ connecting $c_1$ with $M$ such that $\gamma(0)=c_1$ and $\gamma(1)=M$. Let $\gamma_+$ be a new curve defined as $\gamma_+(t)=\gamma(t)$ if $\textup{Im\,}\gamma(t)\geq 0$ and $\gamma_+(t)=\overline{\gamma(t)}$ if $\textup{Im\,}\gamma(t)<0$. Then $\gamma_+$ is a piecewise smooth curve which is disjoint with the lower-half plane. Moreover, $\gamma_+$ is contained in $B_\infty$ since the Fatou set of $f_\lambda$ is symmetric about the real line. Note that $\gamma_+$ is compact and $B_\infty$ is open, one can move $\gamma_+$ slightly in $B_\infty$ (but keep the two ends fixed) such that the new curve $\gamma_+'$ is contained in the upper-half plane (except two ends). Then $\beta:=\gamma_+'\cup\overline{\gamma}_+'$ is a Jordan curve contained in $B_\infty$, $\beta\cap\R=\{c_1,M\}$ and the bounded component of $\C\setminus\beta$ contains a repelling fixed point $x_1\in J(f_\lambda)$. Since $v_0=-\lambda\not\in\MA(\infty)$, it follows that $(-\infty,v_0]\cap J(f_\lambda)\neq\emptyset$. Since $v_0<c_1$, it means that $\beta$ separates $J(f_\lambda)$. This is a contradiction since $J(f_\lambda)$ is connected by Theorem \ref{main-thm} (2). Therefore, we have $v_1\not\in B_\infty$.
\end{proof}

\section{Both free critical values do not escape}\label{main-thmape}

In this section, we consider the case that $v_0$ and $v_1$ do not belong to $\MA(\infty)$. Since the immediate super-attracting basin of $\infty$ is simply connected, it follows that each Fatou component of $\MA(\infty)$ is also simply connected.

\begin{lema}\label{lema-our-case}
Suppose that $v_1\not\in B_\infty$. Then we have \eqref{equ-a-domain-1} if $\kappa=1/5$.
\end{lema}

\begin{proof}
If $v_1$ is not contained in the immediate super-attracting basin of $\infty$, it follows from Lemma \ref{lema-domain-a-2} that $\lambda$ must satisfy \eqref{equ-a-domain-1} if $\kappa=1/5$.
\end{proof}

\begin{proof}[{Proof of Theorem \ref{main-thm} (3)}]
Since $v_1\not\in\MA(\infty)$, we know that $v_0$ is contained in an attracting basin $B_0$ which is completely invariant under $f_\lambda$ by Lemmas \ref{lema-our-case} and \ref{lema-domain-a-1}.

(3a) By the hypothesis, each Fatou component of $f_\lambda$ contains at most one critical value and there is at most one critical orbit in the Julia set. By Lemmas \ref{CFPCS} and \ref{lema-no-Herman}, this means that $J(f_\lambda)$ is connected.

(3b) If $v_0$ and $v_1$ are both contained in $B_0$, then the Fatou set of $f_\lambda$ is equal to $B_0\cup \MA(\infty)$ and the Julia set $J(f_\lambda)=\partial B_0=\partial \MA(\infty)$. Let $D_\lambda:=\D(-\lambda ,\kappa|\lambda|)$ be an open disk with $\kappa=1/5$ and denote $V:=\EC\setminus\overline{D}_\lambda$. According to Lemma \ref{lema-domain-a-1} (2), the preimage $U:=\EC\setminus f_\lambda^{-1}(\overline{D}_\lambda)$ is a Jordan domain. Moreover, $(U,V,f_\lambda)$ is a polynomial-like mapping\footnote{Although $U$ and $V$ are not domains in $\C$, one can use a coordinate transformation to obtain a polynomial-like mapping since as a rational map, $f_\lambda$ is holomorphic on whole $\EC$.} with degree $2n$ and $\EC\setminus V=\overline{D}_\lambda$ does not contain any critical points. By Theorem \ref{straightening}, $(U,V,f_\lambda)$ is hybird equivalent to a polynomial $g$ with degree $2n$. Note that $f_\lambda$ has a super-attracting fixed point at $\infty$ with local degree $n$ and the free critical points $c_1$, $\cdots$, $c_n$ are attracted by the basin $B_0$. This means that $g$ has also a super-attracting fixed point $z_0$ with local degree $n$ and the rest $n$ critical points of $g$ are escaped to $\infty$. In particular, the Julia set of $g$ is disconnected.

Let $K_0$ be the filled Julia component of $g$ which contains the super-attracting fixed point $z_0$. According to \cite[Main Theorem]{QY}, all the Julia components of $g$ are points except the components that are iterated onto the component $K_0$. We claim that $K_0$ is a closed Jordan disk. Indeed, by the hyperbolicity of $g$, one can construct a polynomial-like mapping $(U',V',g)$ with degree $n$ such that $\overline{K}_0\subset U'$, $\overline{U}'\subset V'$ and the filled Julia set of $(U',V',g)$ is exactly $K_0$. Note that $(U',V',g)$ is hybird equivalent to $P_n(z)=z^n$ whose filled Julia set is the closed unit disk. This means that $K_0$ is actually a closed Jordan disk. Therefore, the Julia set of $g$ is the union of countably many Jordan curves and uncountably many points. Since $(U,V,f_\lambda)$ is hybird equivalent to $g$ and $\EC\setminus V=\overline{D}_\lambda$ is contained in the attracting basin $B_0$, we know that the Julia set of $f_\lambda$ is homeomorphic to that of $g$. This means that $J(f_\lambda)$ consists of countably many Jordan curves and uncountably many points.
\end{proof}

\begin{exam}\label{exam-2}
(i) Let $n=3$ and $\lambda=\sqrt{6}/9$. Then $v_0$, $v_1\not\in\MA(\infty)$ and they are in different Fatou components;

(ii) Let $n=3$ and $\lambda=4/25$. Then $v_0$, $v_1\not\in\MA(\infty)$ and both of them are contained in an immediate attracting basin of $f_\lambda$.
\end{exam}

See Figure \ref{Fig_Julia-thm3} for the Julia sets of Example \ref{exam-2}.

\begin{figure}[!htpb]
  \setlength{\unitlength}{1mm}
  \centering
  \includegraphics[width=70mm]{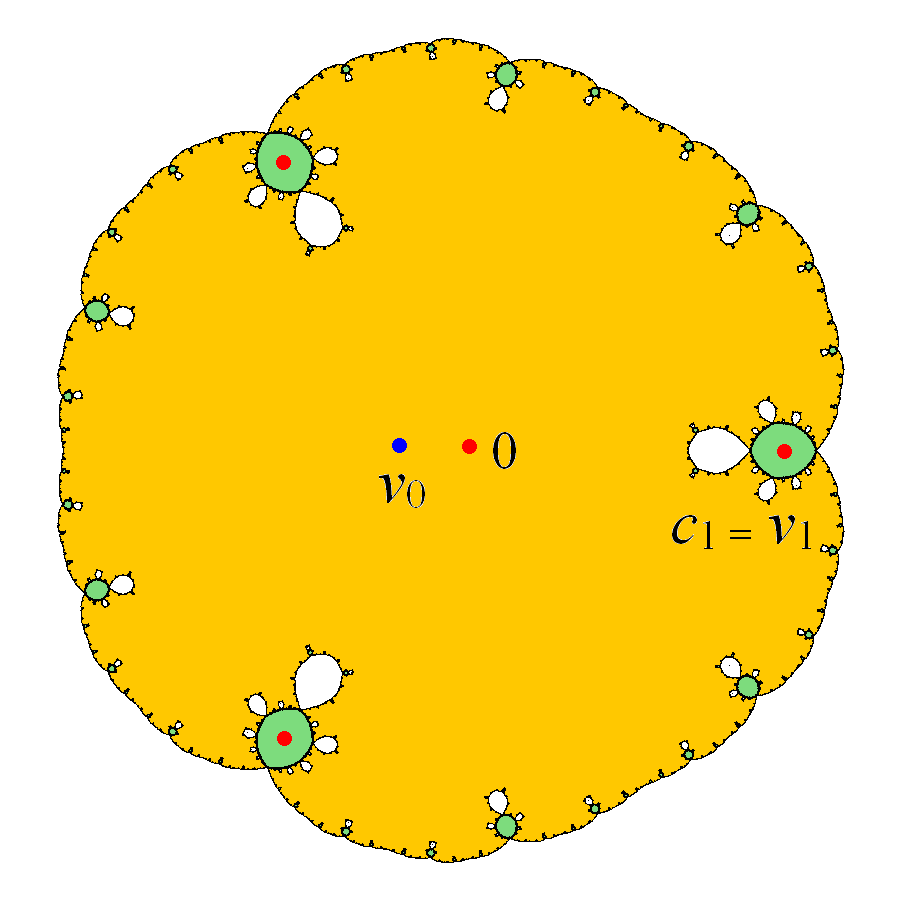}
  \includegraphics[width=70mm]{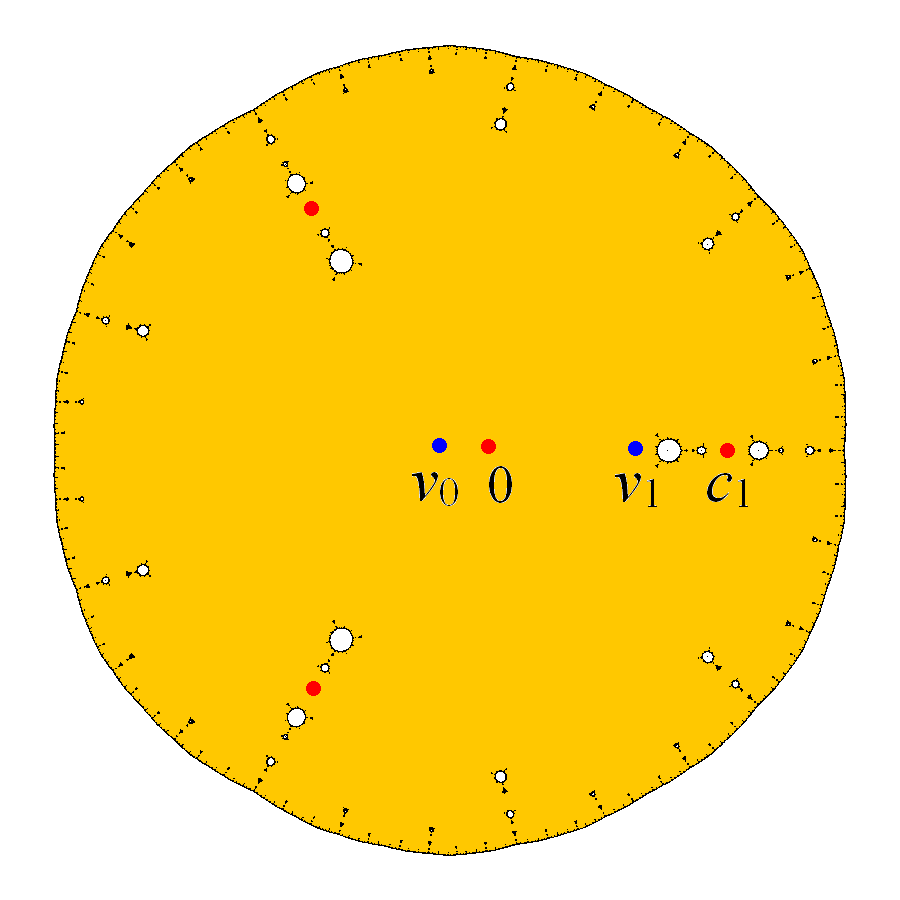}
  \caption{The Julia sets of $f_\lambda$ with different parameters $\lambda_3=\sqrt{6}/9$ and $\lambda_4=4/25$ (from left to right), where $n=3$. The parameter $\lambda_3$ is chosen such that $f_\lambda$ has three (super) attracting basins while $\lambda_4$ is chosen such that $v_0$ and $v_1$ lie in a same Fatou component. These two Julia sets correspond to the two cases that stated in Theorem \ref{main-thm} (3). As in Figure \ref{Fig_Julia-thm2}, the free critical points and values are marked by red and blue dots respectively.}
  \label{Fig_Julia-thm3}
\end{figure}

\begin{proof}
(i) If $n=3$ and $\lambda=\sqrt{6}/9$, then we have $c_1=\sqrt[3]{2\lambda}=3\lambda=v_1$. Therefore, $v_1$ is a super-attracting fixed point of $f_\lambda$. If $\lambda=\sqrt{6}/9$, by \eqref{equ-cal-1} and Lemma \ref{lema-domain-a-1}, there exists an attracting basin $B_0$ containing $\overline{\D}(-\lambda,|\lambda|/5)$ which is invariant under $f_\lambda$. Let $B_1$ be the super-attracting basin of $v_1$. Then $B_0\cap B_1=\emptyset$ since $v_1=3\lambda\not\in \overline{\D}(-\lambda,|\lambda|/5)$. Therefore, $v_0$, $v_1\not\in\MA(\infty)$ and they are in different Fatou components.

(ii) Since $5/12>4/25$, by \eqref{equ-cal-1} and Lemma \ref{lema-domain-a-1}, there exists an attracting basin $B_0$ containing $\overline{\D}(-\lambda,|\lambda|/5)$ which is completely invariant under $f_\lambda$. Therefore, it is sufficient to prove that the forward orbit of $v_1$ is contained in $\overline{\D}(-\lambda,|\lambda|/5)$. If $n=3$ and $\lambda=4/25$, a direct calculation shows that $|f_\lambda^{\circ 2}(v_1)-(-\lambda)|/|\lambda|=0.125\cdots<1/5$. The proof is complete.
\end{proof}

\section{The impossible types of Julia sets}\label{sec-impo}

As stated in the introduction, it was known that the Cantor circles Julia sets and Sierpi\'nski curves Julia sets can appear in McMullen family and the generalized McMullen family. We will prove in the present section that these two kind of Julia sets are not exist for the family $f_\lambda$. Moreover, we also prove that $f_\lambda$ has no Herman rings.

\begin{lema}\label{lema-sier-no}
The Julia set of any polynomial can never be a Sierpi\'nski curve.
\end{lema}

\begin{proof}
Let $P$ be a polynomial with degree at least two. Then $P$ has a super-attracting fixed point at $\infty$. Moreover, the basin $\MA(\infty)$ containing $\infty$ is completely invariant. Therefore, we have $J(P)=\partial\MA(\infty)$. If $P$ has no bounded Fatou components, then $P$ has exactly one Fatou component $\MA(\infty)$ and $J(P)$ cannot be a Sierpi\'nski curve. If $P$ has a bounded Fatou component $U$, then $\partial U\subset J(P)=\partial\MA(\infty)$. This also contradicts with the definition of the Sierpi\'nski curve.
\end{proof}

As an immediate corollary of Theorem \ref{straightening} and \ref{lema-sier-no}, we have

\begin{cor}\label{cor-sier-no}
The Julia set of any polynomial-like mapping can never be a Sierpi\'nski curve.
\end{cor}

\begin{proof}[{Proof of Theorem \ref{thm-no-type}}]
By Theorem \ref{main-thm}, we only need to consider cases (2) and (3).

(i) \textit{The non-existence of Cantor circles}. By definition, a Cantor circles Julia set consists uncountable many Jordan curves. Therefore, $J(f_\lambda)$ is not a Cantor set of circles by Theorem \ref{main-thm} (2) and (3).

(ii) \textit{The non-existence of Sierpi\'nski curves}. Since any Sierpi\'nski curve is connected, we only need to consider cases (2) and (3a). Suppose that $v_1$ is contained in the immediate basin $B_\infty$ of $\infty$. From the proof of Theorem \ref{main-thm} (2), one can construct a polynomial-like mapping $(U,V,f_\lambda)$ such that the Julia set of $(U,V,f_\lambda)$ is homeomorphic to that of $f_\lambda$. By Corollary \ref{cor-sier-no}, $J(f_\lambda)$ is not a Sierpi\'nski curve.

Suppose that $v_1\not\in B_\infty$. By Lemmas \ref{lema-our-case} and \ref{lema-domain-a-1}, one can construct a polynomial-like mapping $(U,V,f_\lambda)$ as in the proof of Theorem \ref{main-thm} (3b), where $V:=\EC\setminus\overline{D}_\lambda$, $U:=\EC\setminus f_\lambda^{-1}(\overline{D}_\lambda)$ and $D_\lambda:=\D(-\lambda ,|\lambda|/5)$. By the choice of the disk $D_\lambda$, the Julia set of $(U,V,f_\lambda)$ is quasi-conformally homeomorphic to that of $f_\lambda$. According to Corollary \ref{cor-sier-no}, $J(f_\lambda)$ is not a Sierpi\'nski curve.

(iii) \textit{The non-existence of Herman rings}. The Julia set of a rational map having a Herman ring is disconnected. So we only need to consider case (3b). However, in case (3b), all the critical points of $f_\lambda$ are contained in the Fatou set. This means that $f_\lambda$ has no Herman rings.
\end{proof}

\section{The case for $n=2$}

In this section, we make some brief comments on the family $f_\lambda$ with $n=2$, i.e.
\begin{equation*}
f_\lambda(z)= z^2+\frac{\lambda^2}{z^2-\lambda}, \text{ where } \lambda\in\C^*.
\end{equation*}
We can obtain the same results on the symmetric properties of the dynamical behaviors as proved in \S\ref{Prelim}. However, Lemmas \ref{lema-domain-a-1} and \ref{lema-domain-a-2} become invalid when $n=2$. Define
\begin{equation*}
\Lambda_0:=\{\lambda\in\C^*:v_0\not\in\MA(\infty)\} \text{ and } \Lambda_1:=\{\lambda\in\C^*:v_1\not\in\MA(\infty)\}.
\end{equation*}
If $n\geq 3$, then $\Lambda_1$ is compactly contained in $\Lambda_0$ by Lemmas \ref{lema-domain-a-1} and \ref{lema-domain-a-2}. However, if $n=2$, we have $(\partial\Lambda_0)\cap\Lambda_1\neq\emptyset$. See Figure \ref{Fig_parameter-n-2}.

\begin{figure}[!htpb]
  \setlength{\unitlength}{1mm}
  \centering
  \includegraphics[height=65mm]{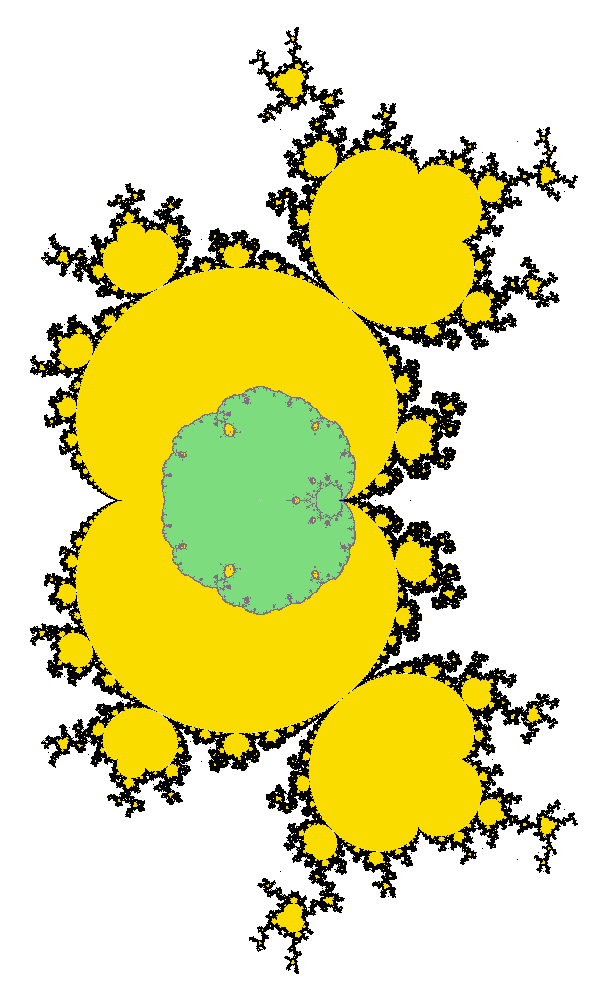}\quad
  \includegraphics[height=65mm]{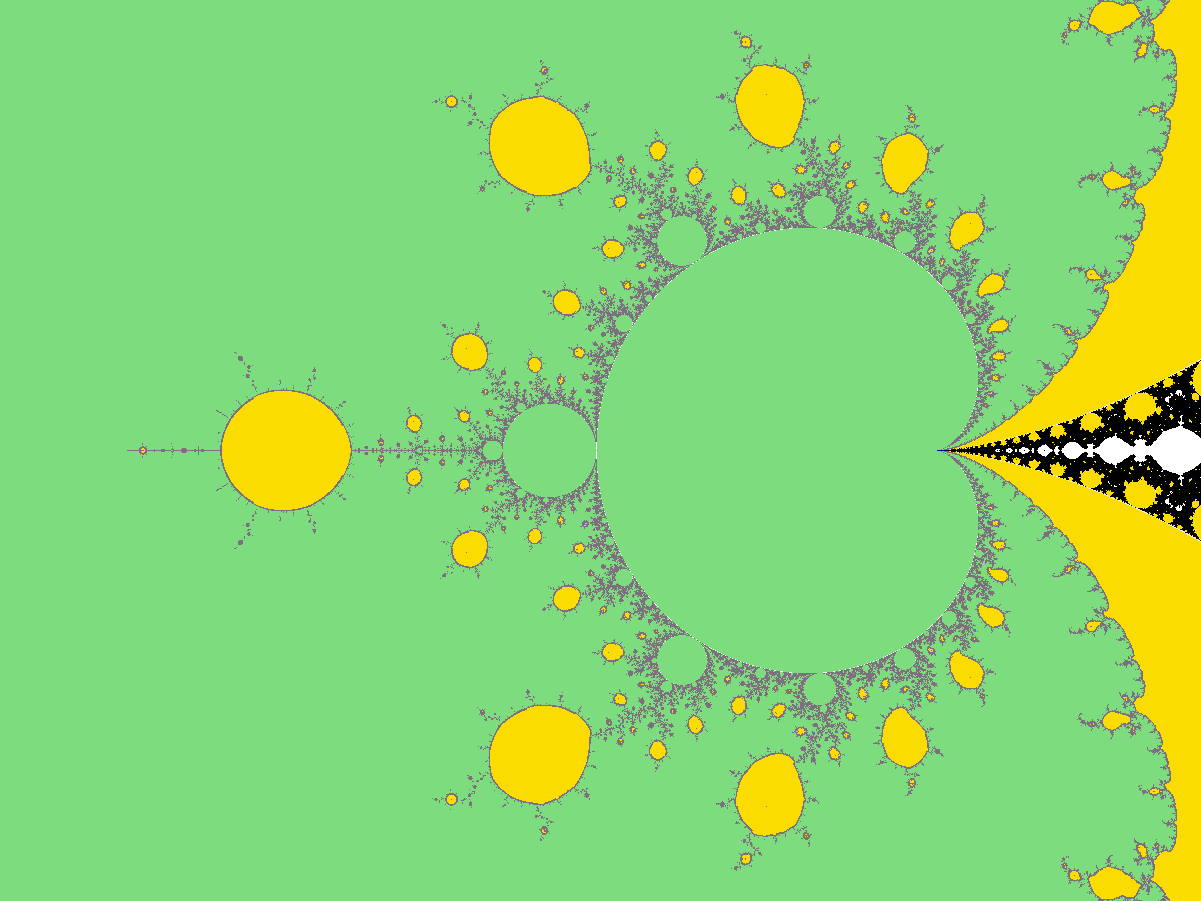}
  \caption{The parameter plane of $f_\lambda$ with $n=2$ and its zoom near $\lambda=1/4$. It can be seen from these two pictures that the intersection of the boundary of $\Lambda_0$ (the yellow part) and $\Lambda_1$ (the green part) is non-empty.}
  \label{Fig_parameter-n-2}
\end{figure}

Actually, we can prove

\begin{lema}
If $n=2$, then $1/4\in(\partial\Lambda_0)\cap\Lambda_1$.
\end{lema}

\begin{proof}
Indeed, if $n=2$ and $\lambda=1/4$, solving the equation
\begin{equation*}
f_\lambda(z)= z^2+\frac{1/16}{z^2-1/4}=z,
\end{equation*}
we obtain two fixed points $z_1=(1+\sqrt{5})/4$ and  $z_2=(1-\sqrt{5})/4$, and both of them have multiplicity two. This means that both $z_1$ and $z_2$ are parabolic fixed points of $f_\lambda$ with multiplier $1$. Since each fixed parabolic basin attracts at least one critical value, it follows that $v_0$ and $v_1$ are not attracted by $\infty$. Therefore, $1/4\in\Lambda_0\cap\Lambda_1$. It is sufficient to prove that $1/4\in\partial\Lambda_0$. Similar to the case of quadratic polynomials $z\mapsto z^2+\lambda$, one can check that $f_\lambda(z)=z$ has no solutions in $\R$ if $\lambda>1/4$. This means that $v_1\in\MA(\infty)$ if $\lambda>1/4$ and hence $1/4\in\partial\Lambda_0$. We omit the details here.
\end{proof}

We conjecture that $(\partial\Lambda_0)\cap\Lambda_1=\{1/4\}$ and $\Lambda_1\subset\Lambda_0$. See Figure \ref{Fig_Julia-n-2} for the Julia set of $f_\lambda$ with $\lambda=1/4$ and $n=2$.

\begin{figure}[!htpb]
  \setlength{\unitlength}{1mm}
  \centering
  \includegraphics[height=65mm]{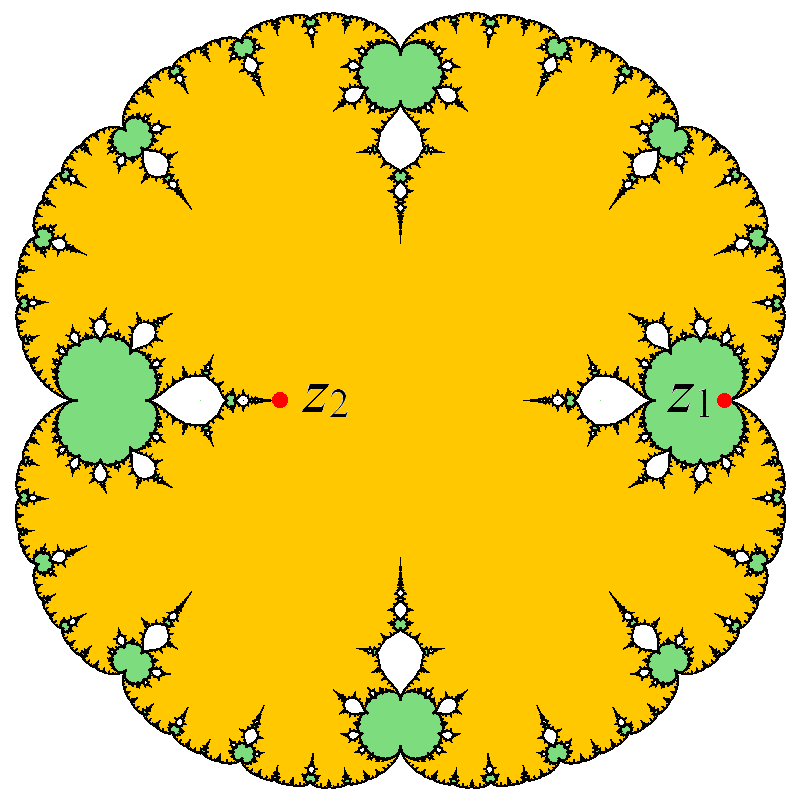}\quad
  \includegraphics[height=65mm]{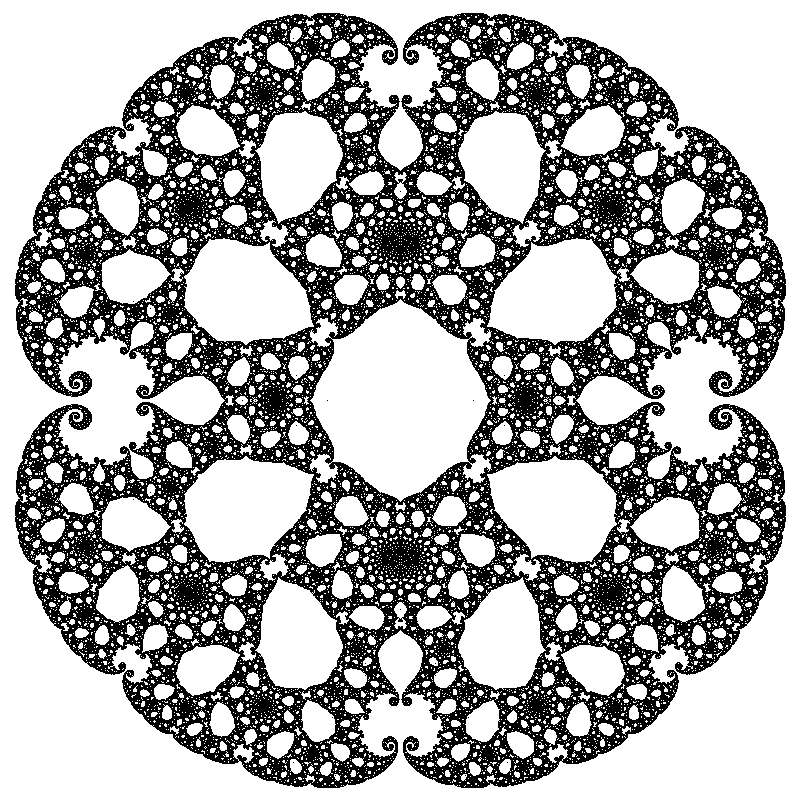}
  \caption{The Julia set of $f_\lambda$ with $\lambda=1/4$ and its perturbation ($\lambda=0.253$), where $n=2$. The Julia set on the left has two parabolic fixed points with multiplier $1$. This is a special example that cannot happen for $f_\lambda$ with $n\geq 3$. The Julia set on the right is a Cantor set. }
  \label{Fig_Julia-n-2}
\end{figure}

A possible method to prove Theorem \ref{main-thm} with $n=2$ is to find a ``nice" curve $\gamma$ separating $\partial\Lambda_0$ and $\Lambda_1$ (the point $1/4$ is on this curve) such that if the parameter is chosen in the unbound component of $\C^*\setminus\gamma$, then $v_1$ is attracted by $\infty$ while $v_0$ is not if the parameter is chosen in the bound component of $\C^*\setminus\gamma$. This strategy is a bit similar to the ideas in Lemmas \ref{lema-domain-a-1} and \ref{lema-domain-a-2}.


\end{document}